\documentclass{amsart}
\usepackage[utf8]{inputenc}
\usepackage{amssymb}
\begin{document}

\hyphenation{def-i-ni-tion par-ti-tion par-ti-tion-able Def-i-ni-tion}

\title[Corrigendum Discr.\ Math.\ 22 (1978) 263]{Corrigendum to ``Paths and circuits in finite groups'', Discr.\ Math.\ 22 (1978) 263}

\author{Maximilian F. Hasler}
\urladdr{http://www.univ-ag.fr/~mhasler}
\email{mhasler@univ-ag.fr}
\address{CEREGMIA \& D.S.I., Universit\'e des Antilles, B.P. 7209, 97275 Schoelcher cedex, Martinique (F.W.I.)}

\author{Richard J. Mathar}
\urladdr{http://www.mpia.de/~mathar}
\email{mathar@mpia.de}
\address{Max Planck Institute of Astronomy, K\"onigstuhl 17, 69117 Heidelberg, Germany}

\date{\today}

\newtheorem{definition}{Definition}
\newtheorem{remark}{Remark}
\newtheorem{theorem}{Theorem}
\newtheorem{lemma}{Lemma}
\newtheorem{example}{Example}
\def\P[#1]{\par\paragraph{\bf#1.}}
\def\set#1{\left\{#1\right\}}
\def\lcm{\mathrm{lcm}}

\begin{abstract}
We remove 52 from the set of prime partitionable numbers in a paper by Holszty\'nski and Strube (1978), 
which also appears in a paper by Erd\H{o}s and Trotter in the same year.
We establish equivalence between two different definitions in the two papers,
and further equivalence to the set of Erd\H{o}s-Woods numbers.
\end{abstract}

\maketitle

\section{The paper by Holszty\'nski and Strube}

In their paper about paths and circuits in finite groups,
Holszty\'nski and Strube define prime partitionable numbers as follows \cite[Defn. 5.3]{HS78}:

\begin{definition}
\label{D1}
An integer $n$ is said to be prime partitionable if there is a partition
$\{\mathbb{P}_1,\mathbb{P}_2\}$ of the set of all primes less than $n$ such that,
for all natural numbers $n_1$ and $n_2$ satisfying $n_1+n_2=n$, we have
$\gcd (n_1, p_1) > 1$ or $\gcd (n_2, p_2) > 1$, for some $(p_1,p_2)\in \mathbb{P}_1\times \mathbb{P}_2$.
\end{definition}

\begin{remark}
According to this definition, only integers $n\ge4$ can be prime partitionable, because
one cannot have a partition $\{\mathbb{P}_1,\mathbb{P}_2\}$ of the primes less than
$3$, since there is only one.
The condition involving the greatest common divisors (GCD's) can also be written as
``$p_1\mid n_1$ or $p_2\mid n_2$'', since the GCD is necessarily equal to the prime if it is larger than 1.
\end{remark}

\begin{example}
\label{E1}
The smallest prime partitionable number is 16. The set of primes less than 16 is $\mathbb{P} = \{2, 3, 5, 7, 11, 13\}$.
There are two partitions $\mathbb{P}=\{\mathbb{P}_1,\mathbb{P}_2\}$ that demonstrate that 16 is
prime partitionable, these are
\begin{equation}
  \set{ \mathbb{P}_1 = \set{2, 5, 11},~    \mathbb{P}_2 = \set{3, 7, 13}}
\end{equation}
and 
\begin{equation}
  \set{ \mathbb{P}_1 = \set{2, 3, 7, 13}, \mathbb{P}_2 = \set{5, 11}} ~.
\end{equation}
\end{example}

Following Definition~\ref{D1}, the authors give the following (wrong!) list of prime partitionable numbers,
\begin{equation}
  \{ 16, 22, 34, 36, 46, 52, 56, 64, 66, 70, ... \}
\label{wrong}
\end{equation}
which erroneously includes the 52.

The paper by Holszty\'nski and Strube is cited, 
and the wrong list \eqref{wrong} is explicitly reprinted, in a paper of Trotter and Erd\H{o}s \cite{TrottJGT2},
where they show that the Cartesian product of two directed cycles of respective lengths $n_1$ and $n_2$ 
is hamiltonian iff the GCD of the lengths is $d \ge 2$ and can be written as sum
$d=d_1+d_2$, where the $d_i$ are coprime to the respective lengths.

Using two independent computer programs written in Maple and PARI/GP  we verified
that the list of prime partitionable numbers according to Definition~\ref{D1} should rather be equal to
\begin{equation}
\{ 16, 22, 34, 36, 46, 56, 64, 66, 70, 76, 78, 86, 88, 92, 94, 96, 100, ... \}	~.
\end{equation}
(With contemporary computer power it is easier to inspect all possible 2-partitions
of the prime sets
beyond what could be done at the time of publication of the original paper.)

\begin{remark}
\label{Rqplusn}
If $n$ is prime partitionable with a representation $n=q^l+n_2$ where $q^l$ is a prime power with $l\ge 1$,
then at least one of the prime factors of $n_2$ is in the same set $\mathbb{P}_j$
as $q$.
[Proof: If $q$ is in $\mathbb{P}_1$, then $n=q^l+n_2$ is supported by $\gcd(n_1,q)>1$ 
on $\mathbb{P}_1$. The commuted $n=n_2+q^l$ is obviously not supported on $\mathbb{P}_2$
because $q$ is not member of $\mathbb{P}_2$, so it must be supported via $\gcd(n_2,p_1)>1$
through a prime member $p_1$ on $\mathbb{P}_1$; so $n_2$
must have a common prime factor with an element of $\mathbb{P}_1$. Alternatively, if $q$
is in $\mathbb{P}_2$, $n_2$ must have a common prime factor with an element of $\mathbb{P}_2$.
In both cases, one of the prime factors of $n_2$ is in the same partition $\mathbb{P}_j$ as $q$.]
\end{remark}

\begin{remark}
\label{R2}
As a special case of Remark~\ref{Rqplusn},
if $n$ is prime partitionable and $n=q_1+q_2$ with $q_1$ and $q_2$ both prime, 
then $q_1$ and $q_2$ are either both in $\mathbb{P}_1$ or both in $\mathbb{P}_2$. 
\end{remark}

\begin{remark}
As a corollary to Remark~\ref{Rqplusn}, if $n$ is one plus a prime power,
$n$ is not prime partitionable.
\end{remark}

\begin{remark}
\label{R2omega}
If $n$ is prime partitionable with a representation $n=1+q_1^{e_1}q_2^{e_2}$, one plus an integer
with two distinct prime factors, then the two prime factos $q_1$ and $q_2$ are
not in the same $\mathbb{P}_j$. The proof is elementary along the lines of Remark \ref{Rqplusn}.
$n=106=1+3\times 5\times 7$
and
$n=196=1+3\times 5\times 13$
are the
smallest prime partitionable numbers
not of that form.
\end{remark}

We see that $52$ is not prime partitionable given the following contradiction:
\begin{itemize}
\item
$52=5^2+3^3$ which forces $3$ and $5$ into the same $\mathbb{P}_j$
according to Remark~\ref{Rqplusn}: $\mathbb{P}_j=\set{3,5,\ldots}$.
\item
$52=3+7^2$ which forces $3$ and $7$ into the same $\mathbb{P}_j$
according to Remark~\ref{Rqplusn}: $\mathbb{P}_j=\set{3,5,7,\ldots}$.
\item
$52=17+5\times 7$ which requires that $5$ or $7$---which are already in the same $\mathbb{P}_j$---are in
the same $\mathbb{P}_j$ as $17$ according to Remark~\ref{Rqplusn}: $\mathbb{P}_j=\set{3,5,7,17,\ldots}$.
\item
$52=1+3\times 17$ which requires that $3$ and $17$ are in distinct sets
according to Remark~\ref{R2omega}.
\end{itemize}

In overview, we establish the equivalence between the set of prime partitionable numbers
of Definition \ref{D1} and two other sets,
the prime partitionable numbers of Trotter and Erd\H{o}s (Section \ref{sec.TE})
and the Erd\H{o}s--Woods numbers (Section \ref{sec.EW}).

\section{The Paper by Trotter and Erd\H{o}s}
\label{sec.TE}

\subsection{The Trotter--Erd\H{o}s definition}

Although the authors of \cite{TrottJGT2} cite the---at that time unpublished---work
of Holszty\'nski and Strube~\cite{HS78},
they give the following seemingly different definition of ``prime partitionable:'' 

\begin{definition}
\label{D2}
An integer $d$ is prime partitionable if there exist $n_1$, $n_2$ with $d=\gcd(n_1,n_2)$
so that for every $d_1,d_2 > 0$ with $d_1+d_2 = d$ either $\gcd(n_1,d_1)\ne 1$ or $\gcd(n_2,d_2)\ne 1$.
\end{definition}

\begin{example}
$d=16=2^4$ is prime partitionable according to Definition~\ref{D2}
based for example on $n_1=880=2^4\times 5\times 11$ and $n_2=4368=2^4\times 3\times 7\times 13$.
\end{example}

According to this definition, any integer $d < 2$ is ``vacuosly''
prime partitionable 
since there are no representations $d=d_1+d_2$ with  $d_1,d_2 > 0$.
So we tacitly understand $d>1$ in Definition~\ref{D2} when we talk about equivalence with Definition~\ref{D1}.

It is easy to see that
a ``witness pair'' $n_1$ and $n_2$ in Definition~\ref{D2}
has the following properties:
\begin{enumerate}
\def\labelenumi{(P\theenumi)}
\item
$d=\gcd(n_1,n_2)$ implies factorizations $n_1 = dk_1$, $n_2 = dk_2$ with $\gcd(k_1,k_2)=1$.
We still may have $\gcd(d,k_1)>1$ or $\gcd(d,k_2)>1$.
\item
One can construct derived ``witness pairs'' where $n_2$ is replaced by $n_2' = dk_2'$, 
where $k_2'$ is the product of all primes less than $d$ that do not divide $n_1$.
[Keeping $k_2'$ coprime to $n_1$ ensures that $\gcd(n_1,n_2')$ still
equals $d$.
The construction also ensures that the prime factor
set of $n_2'$ is a superset of the prime factor set of $n_2$,
so the $d_2$ in Definition~\ref{D2}
which need a common prime factor with $n_2$ still find that common prime factor with $n_2'$.]
$n_1$ may be substituted in the same manner.
\item
One can reduce a witness pair by striking prime powers
in $k_1$ and $k_2$ until $k_1$ and $k_2$
are square-free,
coprime to $d$, and contain only prime factors less than $d$.
\item
The integer $dk_1k_2$
does \emph{not} need to have a prime factor set that contains
all primes $\le d$.
(Example: $d=46=2\times 23$ is prime partitionable established by $k_1=3\times 19\times 37\times 43$
and
$k_2=5\times 7\times 11\times 13\times 17 \times 29\times 41$, where the prime $31$ is not
in the prime factor set of $dk_1k_2$).
The missing prime factors may be distributed across the $n_1$ and $n_2$ (effectively
multiplying $k_1$ or $k_2$ but not both with a missing prime) to generate ``fatter'' prime witnesses.
[$\gcd(n_1,n_2)$ does not change if either $n_1$ or $n_2$ is multiplied
with absent primes. Furthermore the enrichment of the prime factor sets of the $n_j$ 
cannot reduce the common prime factors in $\gcd(d_j,n_j)$. So the witness status is preserved.]
\item
Prime factors of $n_1$ or $n_2$ larger than $d$ have no impact on the
conditions $\gcd(n_j,d_j)\neq 1$ and can as well be removed from $n_j$ without
 compromising the witness status.
\end{enumerate}

\subsection{Equivalence With The Holszty\'nski--Strube Definition}
\begin{lemma}
If a number is prime partitionable according to Definition~\ref{D1}
it is prime partitionable according to Definition~\ref{D2}.
\end{lemma}

\begin{proof}
Assume that a partition $\{\mathbb{P}_1, \mathbb{P}_2\}$ 
establishes that $d$ is prime partitionable according to Definition~\ref{D1}.
Let $k_j$ be the product of all $p \in \mathbb{P}_j$ which do not divide $d$.
Then the pair $n_1=dk_1$, $n_2=dk_2$
satisfies the condition of Definition~\ref{D2}.
[Because the $k_j$ and $d$ are coprime, $d=\gcd(n_1,n_2)$. The prime factor sets
of $n_j$ are supersets of $\mathbb{P}_j$, so the $d_j$ find matching prime factors
according Definition \ref{D1}.]
\end{proof}

\begin{lemma}
If a number is prime partitionable according to Definition~\ref{D2}
it is prime partitionable according to Definition~\ref{D1}.
\end{lemma}
\begin{proof}
Assume that $n_1$, $ n_2$ is a witness pair
conditioned as in Definition~\ref{D2},
propped up with with the aid of (P4) if the union of their prime sets is incomplete.
Then let $\mathbb{P}_1$ contain all prime factors of $n_1$ smaller than $d$
and let $\mathbb{P}_2$ contain the prime factors of $n_2/d$ smaller than $d$:
This yields a partition $(\mathbb{P}_1,\mathbb{P}_2)$ with the properties of Definition~\ref{D1}.
Because we have (asymmetrically) assigned all the prime factors of $n_1$ to $\mathbb{P}_1$,
the decompositions of $d$ where $\gcd(d_1,n_1)>1$ find the associated $p_1$ in Definition~\ref{D1}.
The nontrivial part is to show that if $\gcd(d_1,n_1)=1$ and therefore $\gcd(d_2,n_2)>1$,
then $\gcd(d_2,p_2)>1$ for some $p_2 \in \mathbb{P}_2$. Observe that
a prime factor $p_2$ of $\gcd(d_2,n_2)$---known to exist for $\gcd(d_2,n_2)>1$---cannot divide $k_1$
(coprime to $n_2$); it cannot divide $d$ either, else it would also divide $d-d_2 = d_1$ and thus divide
$\gcd(d_1,n_1) = \gcd(d_1, d k_1)$, contradicting the assumption on the $\gcd(d_1,n_1)$.
So $p_2$ cannot divide the product $k_1d=n_1$ and therefore $p_2$ is in $\mathbb{P}_2$.
\end{proof}

We have shown in the two lemmas above that the two sets of prime partitionable numbers are mutually subsets
of each other; therefore they are equal:
\begin{theorem}
\label{T1}
An integer $d>3$ is prime partitionable in the sense of Definition~\ref{D1}
iff it is prime partitionable in the sense of Definition~\ref{D2}.
\end{theorem}

\begin{remark}
This also implies that if a number is \emph{not} prime partitionable according
to one of the definitions, it neither is according to the other.
\end{remark}

This proof establishes a bijection, for any given $d>3$, between (i) 2-partitions of
the primes less than $d$ and (ii) pairs $(n_1,n_2)$ with the properties that
$d=\gcd(n_1,n_2)$, $k_1=n_1/d$ and $k_2=n_2/d$ square-free and coprime
to $d$ and without prime factors $> d$.

\section{Erd\H{o}s--Woods numbers}
\label{sec.EW}

This is the definition given for sequence A059756 of the 
OEIS \cite{sloane}, entitled ``Erd\H{o}s--Woods numbers:''

\begin{definition}
\label{D3}
An Erd\H{o}s--Woods number is a positive integer $w=e_2-e_1$ which
is the length of an interval of consecutive integers
$[e_1,e_2]=\set{k\in\mathbb{N}\mid e_1\le k\le e_2}$ such that every element $k$
has a factor in common with one of the end points, $e_1$ or $e_2$. 
\end{definition}

\begin{example}
The smallest, $w=16$, refers to the interval $[2184, 2185, ..., 2200]$;
others are listed in \cite[A194585]{sloane}.
The end points are $e_1=2184 = 2^3\times 3 \times 7 \times 13$ 
and $e_2=2200 = 2^3 \times 5^2 \times 11$, and each number $2184\le k\le 2200$
has at least one prime factor in the set $\set{2,3,5,7,11,13}$.
\end{example}

Woods was the first to consider this definition \cite{Woods2013},
Dowe proved that there are infinitely many~\cite{Dowe} and
Cegielski, Heroult and Richard have shown that the set is recursive~\cite{CHR}. 

\begin{lemma}
An Erd\H{o}s--Woods number is prime partitionable according to Definition~\ref{D1}.
\end{lemma}
\begin{proof}
Place all prime factors of $e_1$ into a set $\mathbb{P}_1$
and all prime factors of $e_2$ that are not yet in $\mathbb{P}_1$ into $\mathbb{P}_2$.
This guarantees that the two $\mathbb{P}_j$ do not overlap.
Place any primes $<w$
that are missing in one of the $\mathbb{P}_j$.
Taking
$n_1=k-e_1$, $n_2=e_2-k$,
each $k$ defines 1-to-1 a
partition $w=n_1+n_2$.
If $k$ has a common prime factor with $e_1$, say $p_1$, then $p_1\in \mathbb{P}_1$
and then $n_1$ has a non-trivial integer factorization $n_1=p_1(k/p_1-e_1/p_1)$,
which satisfies the $\gcd(w,p_1)>1$ requirement of Definition \ref{D1} via $\mathbb{P}_1$.
If $k$ has a common prime factor with $e_2$, say $p_2$, then $p_2$ is either
in $\mathbb{P}_1$ (as a common prime factor of $e_1$ and $e_2$) or
in $\mathbb{P}_2$.
The subcases where $p_2$ is in $\mathbb{P}_1$ lead again to 
the non-trivial integer factorization $n_1=p_2(k/p_2-e_1/p_2)$, satisfying
the requirement of Definition \ref{D1} via $\mathbb{P}_1$;
the subcases where $p_2$ is in $\mathbb{P}_2$ lead to 
the non-trivial integer factorization $n_2=p_2(e_2/p_2-k/p_2)$, satisfying
Definition \ref{D1} via $\mathbb{P}_2$.
\end{proof}

\begin{lemma}
A prime partitionable number laid out in Definition~\ref{D1} is an Erd\H{o}s--Woods number.
\end{lemma}
\begin{proof}
The prime partitionable number $n$ of Definition~\ref{D1}
defines for each $1\le j<n$ a prime $p_j$ which is
either the associated $p_1\in \mathbb{P}_1$ if $\gcd(j,p_1)>1$
or the associated $p_2\in \mathbb{P}_2$ if $\gcd(n-j,p_2)>1$ (or any of the two
if both exist). The associated Erd\H{o}s--Woods
interval with lower limit $e_1$ exists because the set of modular equations
$e_1+j \equiv 0\pmod {p_j}$,
$1\le j<n$, can be solved with the Chinese Remainder Theorem \cite{OreAMM59,FraenkelPAMS14}. Duplicates of
the equations are removed, so there may be less than $n-1$ equations if some $p_j$ appear more than
once on the right hand sides. [Note that each $p_j$ is associated with
matching $j$ in the equations, meaning for fixed $p_j$ all the $j$ are in the same modulo class
because they have essentially been \emph{fixed} in Definition~\ref{D1} demanding
$j\equiv 0 \pmod {p_j}$.
No contradicting congruences arise in the set of modular equations.]
\end{proof}
Having shown the two-way mutual inclusion in the two lemmas and joining
in Theorem~\ref{T1} yields:
\begin{theorem}
A number is an Erd\H{o}s--Woods number iff it is prime partitionable according to Definition \ref{D1}
(and therefore iff it is prime partitionable according to Definition \ref{D2}).
\end{theorem}

\bibliographystyle{amsplain}
\bibliography{all}
\end{document}